\documentclass[11pt,a4paper]{article}
\usepackage{mathrsfs}
\usepackage{amsmath,amsthm,color,graphicx,eepic}
\usepackage{amssymb} \usepackage{verbatim}
\usepackage{dsfont}

\usepackage{hyperref}

\usepackage{mathtools} 

\newtheorem{theorem}{Theorem}
\newtheorem{corollary}[theorem]{Corollary}
\newtheorem{lemma}[theorem]{Lemma}

\newtheorem{observation}[theorem]{Observation}

\newcommand{\floor}[1]{\left\lfloor{#1}\right\rfloor}
\newcommand{\ceil}[1]{\left\lceil{#1}\right\rceil}
\newcommand{\abs}[1]{\left\vert{#1}\right\vert}
\begin{document}

\title{Sets avoiding a rainbow solution to the generalized Schur equation}
\author{
Ervin Gy\H{o}ri\thanks{HUN‐REN Alfréd Rényi Institute of Mathematics, Budapest, Hungary.} \and
Zhen He\thanks{School of Mathematics and Statistics, Beijing Jiaotong University, Beijing, China.} \and
Zequn Lv\thanks{Department of Mathematical Sciences, Tsinghua University Beijing, China.} \and
Nika Salia\thanks{Department of Mathematics, King Fahd University of Petroleum \& Minerals, Dhahran, Saudi Arabia.} \and
Casey Tompkins\footnotemark[1] \and
Kitti Varga\thanks{Department of Computer Science, Budapest University of Technology and Economics, Budapest, Hungary.} \and
Xiutao Zhu\thanks{School of Mathematics, Nanjing University of Aeronautics and Astronautics, Nanjing, China.}
}
\date{}
\maketitle

\begin{abstract}
A classical result in combinatorial number theory states that the largest subset of  $[n]$ avoiding a solution to the equation $x+y=z$ is of size $\lceil n/2 \rceil$. 
For all integers $k>m$, we prove multicolored extensions of this result where we maximize the sum and product of the sizes of sets $A_1,A_2,\dots,A_k \subseteq [n]$ avoiding a rainbow solution to the Schur equation $x_1+x_2+\dots+x_m=x_{m+1}$.  Moreover, we determine all the extremal families. 
\end{abstract}

\section{Introduction}
Given an equation $E$, a classical problem is to determine the largest subset $S$ of $[n]=\{1,2,\dots,n\}$ for which there is no way to assign elements of $S$ to the variables of $E$ to yield a solution of $E$.
Typically we allow that an element from the set $S$ is assigned to multiple variables.
In the simple case of the equation $x+y=z$, it is well-known that the maximum is $\ceil{n/2}$, which, if $n\geq 10$, is attained by taking $S$ to be either the set of odd numbers in $[n]$ or the $\ceil{n/2}$ largest numbers; a precise characterization for all $n$ of the extremal cases is given in~\cite{chung1997integer}. 
In the case of the equation $x+y=2z$, that is, a $3$-term arithmetic progression, Roth~\cite{roth1953certain} proved that any set with no solution has size $o(n)$ (here, it is required that variables are assigned distinct elements, otherwise the problem is trivial).
Obtaining more detailed estimates for the size of sets avoiding a $3$-term arithmetic progression is a notoriously difficult problem.  

For the equation $x+y=3z$, Chung and Goldwasser~\cite{chung1997integer} proved that the maximum size of a set $S\subseteq [n]$ with no solution to this equation is $\lceil n/2 \rceil$ for $n\ge 5$, thereby settling a conjecture of Erd\H{o}s. Moreover, for sufficiently large $n$ they proved that the unique extremal construction is the set of odd positive integers less than or equal to $n$. 
Chung and Goldwasser~\cite{chung1996maximum} gave a construction for a large subset of $[n]$ with no solution to $x+y=kz$ for a constant integer $k\ge 3$. 
These constructions were found in an investigation of the natural continuous version of this problem wherein one maximizes the Lebesgue measure of a subset of $[0,1]$ which avoids a solution to the equation. 
The constructions consist of a union of a constant number of intervals. 
For $k\ge 4$, they settled this continuous version. The original discrete version was settled for $k\ge 4$ by Baltz, Hegarty, Knape, Larsson, and Schoen~\cite{02221759}.
For the continuous version of case $k=3$, it was shown by Matolcsi and Ruzsa~\cite{matolcsi2013sets} that an upper bound of $\frac{1}{2}-\frac{1}{114}$ holds, which implies that the continuous and discrete versions of the problem are fundamentally different.  

A set of integers which avoids the equation $w+x=y+z$ except in the trivial case $\{w,x\}=\{y,z\}$ is known as a Sidon set.  
Erd\H{o}s and Tur\'an~\cite{erdHos1954problem} showed that the maximum size of a Sidon set $A\subseteq [n]$ has order $\sqrt{n}$. 
The general problem of finding the largest subset of $[n]$ avoiding a given linear equation was initiated by Ruzsa~\cite{ruzsa1993solving, ruzsa2}. 
He proved, among many other results, that for homogeneous linear equations $E$:~$a_1x_1+\cdots+a_mx_m=0$ a maximal subset of $[n]$ avoiding $E$ has size $o(n)$ if $\sum_{i=1}^m a_i=0$ and $\Theta(n)$ otherwise.
For equations of the form $ax+by=z$, where $a$ and $b$ are natural numbers, Handcock and Treglown~\cite{hancock2017solution} showed that the largest solution-free subsets have the form $\big[ \lfloor \frac{n}{a+b}\rfloor+1,n \big]$ when $n$ is sufficiently large and $a \ge 2$.
Handcock and Treglown~\cite{hancock2016solution} also attained several results for equations of the form $ax+by=cz$ (see also Hegarty~\cite{hegarty2007extremal} for further results on such equations). 

The equations we consider in the present paper have the form $x_1+x_2+\dots+x_m=x_{m+1}$.
Equations of this type are known as generalized Schur equations and have been studied extensively in Ramsey-type settings (see~\cite{budden2024rainbow} for a recent anti-Ramsey result for such equations). For the extremal problem, a simple lower bound is obtained by the set $\{\floor{n/m}+1,\floor{n/m}+2,\dots,n\}$, which has size $n-\floor{n/m}$. 
This result is well known and is a consequence of some more general results (for example, of Corollary~27 in~\cite{hancock2017solution}).  

In this paper, we study a so-called cross~\cite{hilton1977intersection}, multicolor~\cite{bollobas2004multicoloured,keevash2004multicolour} 
or rainbow~\cite{chakraborti2022rainbow,gyHori2022some} version of the problem of finding a maximum solution-free set for a given equation.
In particular, for the generalized Schur equation $x_1+x_2+\dots+x_m=x_{m+1}$, we consider the problems of maximizing $\sum_{i=1}^k |A_i|$ or $\prod_{i=1}^k |A_i|$, where $A_i \subseteq [n]$ for each $i \in [k]$ and there is no solution to the equation in which each of the variables of the solution comes from a distinct set $A_i$. 
We call such a solution a rainbow solution. We will make use of the following notation: let $[n]=\{1,2,\dots,n\}$ for any positive integer $n$, and let $[a,b]=\{a,a+1,\dots,b\}$ for any integers $a$ and $b$ with $a<b$.  We refer to the sets $[a,b]$ as intervals.

We prove an optimal upper bound on $\sum_i |A_i|$ and classify all extremal families. Assuming that each $A_i$ is nonempty, the maximum total size is asymptotically $k\Bigl(n-\frac{n}{m}\Bigr)$ as a function of $n$. We show that there are three types of optimal constructions: (i) taking $m$ sets equal to $[n]$ and leaving the remaining sets empty so that the equation is trivially never satisfied; (ii) choosing all $A_i$ to be approximately the interval $\Bigl[\frac{n}{m},n\Bigr]$; and (iii) in the case when $m=2$ and $n$ is odd, letting each $A_i$ consist of the odd integers in $[n]$. The precise result is stated below.

\begin{theorem}\label{Theorem_Main}
Let $m\geq 2$, $k\geq m+1$ and $n\geq m+1$ be integers, and let $\{A_i\}_{i\in[k]}$ be non-empty subsets of $[n]$. Write $n = mq + r$ with integers $q$ and $r$ satisfying $0 \le r < m$.
If there is no rainbow solution to $x_1+x_2+\cdots+x_m=x_{m+1}$, then
\begin{equation*}\label{Equation_Main_Theorem_Upper_bound}
\sum_{i=1}^{k}|A_i|\leq k(n-q)+m-(r+1).    
\end{equation*}
If equality holds, then 
\begin{itemize}
    \item the family $\{A_i\}_{i\in[k]}$ contains $k-m$ sets equal to $[q+1,n]$ and the rest of the sets are equal to $[t_i,n]$ for any $i\in[m]$, such that $\sum_{i=1}^m t_i=n+1$, where $1\le t_i\leq q+1$, or 
    
     \item $k=m+1$, $n=qm$, the family $\{A_i\}_{i\in[m+1]}$ contains $m$ sets equal to $[q,n]$ and the remaining one is equal to $[q+1,n-1]$, or

    \item $m=2$, $n$ is odd, and $A_i= \big\{ x\in[n] \colon \, \text{$x$ is odd} \big\}$ for all $i\in [k]$.
\end{itemize}
\end{theorem}

The proof of the theorem is split into two cases, $m\geq 3$ and $m=2$, each relying on the following lemma. 
Also, this lemma is the key for determining the extremal number and the structure of the extremal families in the case when the condition on the sets $A_i$ being non-empty is dropped.

\begin{lemma}\label{Lemma_monotone}
    If the statement of Theorem~\ref{Theorem_Main} holds for the family of sets $\{A_i\}_{i\in[k]}$ with  the additional condition
    \[
[n] \supseteq A_1 \supseteq A_2 \supseteq \cdots \supseteq A_m \supseteq A_{m+1} = A_{m+2} = \cdots = A_k,
\]
then it holds in general.
\end{lemma}
\begin{proof}
   Let $\{A_i\}_{i\in [k]}$ be a family of subsets of $[n]$ with no rainbow solution to $x_1+x_2+\dots+x_m=x_{m+1}$ which maximizes $\sum_{i=1}^k |A_i|$.
Define an auxiliary family $\{B_i\}_{i\in[k]}$ by
\[
B_i \coloneqq \Bigl\{ x\in [n] : \Bigl|\{ j\in[k] : x\in A_j \}\Bigr| \ge i \Bigr\}.
\]
Then, the family $\{B_i\}_{i\in[k]}$ is a nested family and $\sum_{i=1}^k |A_i| = \sum_{i=1}^k |B_i|$.
Moreover, the family $\{B_i\}_{i\in[k]}$ admits no rainbow solution. Indeed, if it did, one could order the elements of that solution in increasing order according to the index of the $B_i$ from which they originate, and then greedily construct a rainbow solution for $\{A_i\}_{i\in[k]}$ in that order, contradicting our assumption.
Furthermore, if an element is contained in at least $m+1$ sets, we may assume it belongs to all of them. 
If the family $\{B_i\}_{i\in[k]}$ equals one of the families achieving the upper bound in the statement of Theorem~\ref{Theorem_Main}, by careful analysis of the structure it is easy to see that $\{B_i\}_{i\in[k]}=\{A_i\}_{i\in[k]}$.
\end{proof}

Note that Lemma~\ref{Lemma_monotone} implies that allowing empty sets in Theorem~\ref{Theorem_Main}, the extremal families are either the ones listed in Theorem~\ref{Theorem_Main}, or $A_i=[n]$ for each $i\in [m]$ and $A_i=\emptyset$ otherwise. Indeed, if an extremal family $\{ A_i \}_{i\in[k]}$ contains at least $m+1$ non-empty sets, then by Lemma~\ref{Lemma_monotone}, we can assume none of the sets are empty and we are done by Theorem~\ref{Theorem_Main}. If, however, an extremal family $\{ A_i \}_{i\in[k]}$ contains at most $m$ non-empty sets, then we clearly get $A_i=[n]$ for each $i\in [m]$ and $A_i=\emptyset$ otherwise.

Finally, by Theorem~\ref{Theorem_Main}, the following corollary is immediate.
\begin{corollary}
Let $m\geq 2$ and $k\geq m+1$ be integers, and let $\{A_i\}_{i\in[k]}$ be subsets of $[n]$. Write $n = mq + r$ with integers $q$ and $r$ satisfying $0 \le r < m$.
If there is no rainbow solution to $x_1+x_2+\cdots+x_m=x_{m+1}$, then
\[
\prod_{i=1}^k |A_i| \le \left(n-\floor{\frac{n}{m}}+1 \right)^{m-(r+1)}\left(n-\floor{\frac{n}{m}}\right)^{n-m+(r+1)}.
\]
If equality holds, then either the family $\{A_i\}_{i\in[k]}$ contains $m-(r+1)$ sets equal to $[q,n]$ and the rest of the $n-m+(r+1)$ sets are equal to $[q+1,n]$; 
or $m=2$ and $n$ is odd and $A_i=\{x\in[n]: \text{$x$ is odd} \}$ for all $i\in[k]$.
  
\end{corollary}

In the following, we prove Theorem~\ref{Theorem_Main} in two parts: when $m \ge 3$ and when $m = 2$.

\section{Proof of Theorem~\ref{Theorem_Main} for \texorpdfstring{$\boldsymbol{m\geq 3}$}{}}

By Lemma~\ref{Lemma_monotone}, we assume 
\[
[n] \supseteq A_1 \supseteq A_2 \supseteq \cdots \supseteq A_m \supseteq A_{m+1} = A_{m+2} = \cdots = A_k.
\]


Let $t_i = \min A_i$ for each $i \in [k]$. Since the family $\{A_i\}_{i\in[k]}$ is nested, it follows that
\[
t_1 \leq t_2 \leq \cdots \leq t_m \leq t_{m+1} = t_{m+2} = \cdots = t_k.
\]
We now split the proof into three cases based on the values of $t_{m+1}$ and the sum $t_1+t_2+\cdots+t_m$.

\medskip

\noindent \textbf{Case 1:} $t_{m+1} \le q$.

Define the set $X \coloneqq A_1 + t_2 + t_4 + t_5 + \cdots + t_{m+1}$,
and consider the set~$A_3$. Since there is no rainbow solution, we have $X \cap A_3 = \emptyset$, and clearly $X \cup A_3 \subseteq [t_3, n + t_2 + t_4 + t_5 + \cdots + t_{m+1}].$
We obtain
\[
|A_1| + |A_3| = |X| + |A_3| 
\le n + (t_2 - t_3) + t_4 + \cdots + t_{m+1} + 1 
\le n + (m-2)t_{m+1} + 1 
\le n+(m-2)q + 1.
\]
Thus, if $r\neq 0$ or $m\neq 3$,  we deduce 
\begin{align*}
\sum_{i=1}^{k} |A_i|
&\le 
|A_1|+|A_2|+|A_3|+|A_4| +
(k-4)\floor{\frac{|A_1|+|A_3|}{2}}
\leq 2(|A_1|+|A_3|)+ (k-4)\floor{\frac{|A_1|+|A_3|}{2}} \\ 
&\le  2( n+(m-2)q + 1)+ (k-4)\floor{\frac{n+(m-2)q + 1}{2}}
< k(n-q)+ \big( m-(r+1) \big).  
\end{align*}
The verification of the final inequality is straightforward by considering two cases separately: for $r\neq 0$, the equation holds without taking the floor function, and for $r=0$, the inequality holds since $m>3$. Since the final inequality is strict, equality in Theorem~\ref{Theorem_Main} does not occur in this case.

Note that if $r=0$ and $m=3$, we have 
\[
 \sum_{i=1}^{k} |A_i|\leq k(n-q)+ \big( m-(r+1) \big).
\]
Thus either we are done or $\sum_{i=1}^{k} |A_i|=k(n-q)+ \big( m-(r+1) \big)$, which means $|A_1|+|A_3|=|A_2|+|A_4|=n+(m-2)q+1$. Thus $A_1=A_2$ and $A_3=A_4$ since those sets are nested.
Even more, $t_2=t_3$ and $t_4=q$, thus $X \cup A_3 = [t_3, n + t_3 + t_4]=  [t_3, n + t_3+q]$ since $m=3$.
Since $A_3=A_4$, we have $t_2=t_3=t_4=q$. 
Since there is no rainbow solution, we have $t_2+t_3+t_4=3q=n\not\in A_1$, hence $n + t_2 + t_4 = 4q + t_3\notin X\subseteq X\cup A_3$ a contradiction to  $X\cup A_3=[t_3, 4q + t_3]$. 
Thus, if $r=0$ and $m=3$, we have $\sum_{i=1}^{k} |A_i|< k(n-q)+ \big( m-(r+1) \big)$. 

Therefore, the equality in Theorem~\ref{Theorem_Main} does not hold in Case~1.

\bigskip

\noindent \textbf{Case 2:} $t_{m+1}\geq q+1$ and $t_1+t_2+\cdots+t_m\leq q$.

Consider the sets $Y \coloneqq A_1+t_2+t_3+\cdots+t_{m}$ and $A_{m+1}$.
Since there is no rainbow solution, we have $Y \cap A_{m+1} = \emptyset$, and clearly $Y \cup A_{m+1} \subseteq [t_1+t_2+\cdots+t_m, n+t_2+\cdots+t_m]$.
If $k=m+1$, we define $A_{m+2} \coloneqq \emptyset$ for simplicity of notation. Then, we have
\[
|A_2|+|A_{m+2}| \le |A_1|+|A_{m+1}| = |Y|+|A_{m+1}| \leq n.
\]
Thus, 
\[
\sum_{i=1}^{m+2} |A_i| = \big( |A_1|+|A_{m+1}| \big) + \big( |A_2|+|A_{m+2}| \big) + \sum_{i=3}^{m} |A_i| \leq nm. 
\]
So if $k\geq m+2$, then
\[
\sum_{i=1}^{k} |A_i|
\leq k\dfrac{1}{m+2} \sum_{i=1}^{m+2} |A_i|
\leq kn\dfrac{m}{m+2} 
< kn\dfrac{m-1}{m} 
<  k(n-q)+ \big( m-(r+1) \big).
\]

If $k=m+1$, then we have  $\sum_{i=1}^{k} |A_i|\leq nm$, and equality holds if $A_i=[n]$ for all $i\in [m]$. Hence $A_{m+1}=\emptyset$ a contradiction of our assumption that $A_i \not= \emptyset$ for any $i\in [k]$.

Therefore, the equality in Theorem~\ref{Theorem_Main} does not hold in Case~2.

\bigskip

\noindent \textbf{Case 3:} $t_{m+1} \geq q+1$ and  $t_1+t_2+\cdots+t_m \geq q+1$.

Consider the sets $Z\coloneqq A_1+A_2+\cdots+A_{m}$ and $A_{m+1}$.
Since there is no rainbow solution, we have $Z \cap A_{m+1} = \emptyset$, and clearly
$Z \cup A_{m+1} \subseteq [q+1,mn]$.
By an elementary argument in additive combinatorics (see, for example, Theorem~2.3 in~\cite{glasscock2011sumset}), we know that 
\[
|Z| = |A_1+A_2+\cdots+A_{m}|\geq |A_1|+|A_2|+\cdots+|A_{m}|-(m-1),
\]
where equality holds if and only if all sets are arithmetic progressions with the same difference.
Thus, we have
\[
\sum_{i=1}^{m+1} |A_i| 
\leq |Z|+(m-1)+|A_{m+1}|
\leq mn - q + (m-1),
\]
and equality holds if and only if $Z \cup A_{m+1}= [q+1,mn]$ and all the sets $A_1, \ldots, A_m$ are arithmetic progressions with the same difference.
So we have
\begin{align*}
\sum_{i=1}^{k} |A_i|
&\leq   \sum_{i=1}^{m+1} |A_i|+ \big( k-(m+1) \big) \floor{\dfrac{1}{m+1} \sum_{i=1}^{m+1} |A_i|}\\
&\leq  mn - q + (m-1)+ \big( k-(m+1) \big) \floor{\dfrac{mn - q + (m-1)}{m+1}}\\
&= mn - q + (m-1)+ \big( k-(m+1) \big)(n-q)
= k(n-q)+ \big( m-(r+1) \big).    
\end{align*}

Here, we have established the upper bound. 
It remains to characterize the structure of the family $\{A_i\}_{i\in[k]}$ when $\sum_{i=1}^k |A_i|$ attains this maximum. Under the assumption that the sets are nested and equality holds, we deduce that the sets $A_1, \ldots, A_m$ are arithmetic progressions with the same difference. We claim that for each $i\in[m]$ the set $A_i$ must have the same difference $1$; otherwise, the total size would be at most $\frac{(n+1)k}{2}<k(n-q)+ \big( m-(r+1) \big)$, a contradiction. Moreover, since equality holds, we have
$Z \cup A_{m+1} = [q+1,mn]$,
so $n\in A_i$ and $A_i=[t_i,n]$ for all $i\in[m]$, thus $Z=[t_1+\cdots+t_m,nm]$
and $A_{m+1}=[q+1,t_1+\cdots+t_m-1]$.

If $\sum_{i=1}^m t_i\geq n+1$, then from the maximality, we obtain $\sum_{i=1}^m t_i= n+1$. Clearly, $t_1 \le \dots \le t_m \le t_{m+1} = q+1$ and $A_{m+1} = \ldots = A_k =[q+1,n]$.  Thus, we have the first extremal family described in Theorem~\ref{Theorem_Main}. 


Let us consider the case when $\sum_{i=1}^m t_i\leq n$. Then, we have $\sum_{i=1}^{m-1} t_i + (q+1) > n$ since otherwise $A_1 = [t_1, n]$ would imply that $a_1 + \sum_{i=1}^{m-1} t_i + (q+1) = n$ holds for some $a_1 \in A_1$, contradicting that there is no rainbow solution. Thus, $t_m \le q+1$, so $t_1 \le \ldots \le t_m \le q$. Hence $mq \le n \le \sum_{i=1}^{m-1} t_i + q \le mq$, which implies $n = mq$ and $t_1 = \ldots = t_m = q$. This gives the second extremal family.

This completes the proof of Theorem~\ref{Theorem_Main}.

\section{Proof of Theorem~\ref{Theorem_Main} for \texorpdfstring{$\boldsymbol{m = 2}$}{}}

By Lemma~\ref{Lemma_monotone}, we assume 
\[
[n] \supseteq A_1 \supseteq A_2 \supseteq \cdots \supseteq A_m \supseteq A_{m+1} = A_{m+2} = \cdots = A_k.
\]
Thus, if Theorem~\ref{Theorem_Main} holds for $k=3$ then it holds for $k>3$.
From here, we assume $k=3$ and consider two cases.

\noindent \textbf{Case 1:} there exists $a_3\in [n]$ such that  $a_3$, $n-a_3 \in A_3$.

For any $a_1 \in A_1$ and $a_2\in A_2$, we have $a_1 \not\equiv a_2-a_3 \pmod{n}$ since otherwise there is a rainbow solution $a_1+a_3=a_2$ or $a_2+(n-a_3)=a_1$. 
Thus, for every element $a_1\in A_1$,  either $1\leq a_1+a_3\leq n$ or $1\leq a_1+a_3-n\leq n$, so in both cases $a_1+a_3$ and $a_1+a_3-n$ are not elements of $A_2$. 
Hence we have $|A_1|+|A_2|\leq n$, which implies
\[
|A_1|+|A_2|+|A_3| \le |A_1|+|A_2|+|A_2| \le |A_1|+|A_2|+ \frac{|A_1|+|A_2|}{2} \leq n+ \frac{n}{2}<n+1+\ceil{\frac{n}{2}}.
\]
Therefore, the equality in Theorem~\ref{Theorem_Main} does not hold in this case.

\bigskip

\noindent \textbf{Case 2:} for every $a_3 \in A_3$, we have $n-a_3 \notin A_3$.

Let $t_3 \coloneqq \min A_3$ and $M \coloneqq \max A_3$, and let $d \coloneqq \min (t_3, n-M )$. Then, we have $A_3 \subseteq [d,n-d]$.
By the assumption of the case, we have 
\begin{equation}\label{equation:m=3_A_3}
 |A_3| \leq \floor{\frac{n+1-2d }{2}}= \ceil{\frac{n}{2}}-d.
\end{equation}

The sets $A_1+t_3$ and $A_2$ are disjoint since no rainbow solution exists, and $(A_1+t_3) \cup A_2 \subseteq [1,n+t_3]$. Thus, we have
\begin{equation}\label{Equation:A_1+A_2_1}
   |A_1|+|A_2| = |A_1+t_3|+|A_2|\leq n+t_3<n+t_3+1. 
\end{equation}
The sets $A_1+n-M$ and $n-A_2$ are also disjoint since no rainbow solution exists, and $(A_1+n-M) \cup (n-A_2) \subseteq [0,2n-M]$. 
Thus, we have
\begin{equation}\label{Equation:A_1+A_2_2}
|A_1|+|A_2| = |A_1+n-M|+|n-A_2|\leq 2n - M + 1 = n+1+(n-M).
\end{equation}
By~\eqref{Equation:A_1+A_2_1} and~\eqref{Equation:A_1+A_2_2}, we have 
\[
|A_1|+|A_2| \leq n+1+\min (t_3,n-M) = n+1+d.
\]
Therefore, 
\begin{equation}\label{Equation:m=3_final}
  |A_1|+|A_2|+|A_3| \leq (n+1+d) + \left\lceil \frac{n}{2} \right\rceil - d = n + 1 + \left\lceil \frac{n}{2} \right\rceil.
\end{equation}

We have finished the proof of the upper bound. 
From now on, we assume $|A_1|+|A_2|+|A_3|=n+1+\ceil{\frac{n}{2}}$ and deduce the structure of the sets $A_i$.
Since equality holds in~\eqref{Equation:m=3_final}, we have equality in~\eqref{equation:m=3_A_3} and either in~\eqref{Equation:A_1+A_2_1} or  in~\eqref{Equation:A_1+A_2_2}. 
As the inequality in~\eqref{Equation:A_1+A_2_1} is strict, we have equality in~\eqref{Equation:A_1+A_2_2}. 
Moreover, since~\eqref{Equation:A_1+A_2_1} contains a strict inequality, we have $n-M<t_3$, and so $d=n-M$.
As equality holds in~\eqref{Equation:A_1+A_2_2}, we have 
\[
(A_1+n-M) \cup (n-A_2) = [0,2n-M].
\]
This implies $[M,n]\subseteq A_2\subseteq A_1$.

\begin{observation}\label{Observation:A_3iff}
As in \eqref{equation:m=3_A_3} equality holds, we have $A_3\subseteq [t_3,M]\setminus\{\frac{n}{2}\}$  and  for every $x \in[d,M]\setminus \{\frac{n}{2}\}$ we have  $x \in A_3$ if and only if $n-x \notin A_3$. Moreover, for every $a_3 \in A_3$, we have $n-a_3 \notin A_1$ as $n\in A_2$. 
\end{observation}
From here, we split the proof into two cases depending on whether $t_3<\frac{n}{2}$ or $t_3> \frac{n}{2}$. Note that $t_3 \ne \frac{n}{2}$ since $t_3\in A_3$.

First, assume that $t_3<\frac{n}{2}$. By Observation~\eqref{Observation:A_3iff}, and as $n-M<t_3$, we have $[n-t_3+1,M]\subseteq A_3$. Then $A_3 \subseteq A_2$ and $[M,n] \subseteq A_2$ implies $[n-t_3+1,n] \subseteq A_2$. Hence, since there is no rainbow solution and $t_3 \in A_3$, we have $[n-2t_3+1,n-t_3]\cap A_1=\emptyset$. Also, $t_3 + t_3 < n - t_3 + 1$ as $t_3 \in A_1$, implying $t_3 + 1 \le n - 2t_3 + 1$.
Thus, by Observation~\eqref{Observation:A_3iff}, as $[n-2t_3+1,n-t_3]\cap A_3=\emptyset$, we have $[t_3, 2t_3-1]\subseteq A_3$. Note that this also implies $4t_3-1\leq n$.
We have $[1,t_3-1]\cap A_1=\emptyset$ since otherwise if $a_1\in [1,t_3-1]\cap A_1$, then we have $a_1+t_3\in [t_3+1,2t_3-1]\subseteq A_2$, resulting in a rainbow solution.
Clearly, Observation~\eqref{Observation:A_3iff} implies $A_3 \cap [2t_3, n-2t_3] = \lfloor \frac{n - 4t_3 + 1}{2} \rfloor$, and therefore $A_1 \cap [2t_3, n-2t_3] = \lfloor \frac{n - 4t_3 + 1}{2} \rfloor$.

Now we are ready to estimate $|A_1|+|A_2|+|A_3|$, by considering the following partition
$[n]=[1,t_3-1]\cup[t_3,2t_3-1]\cup[2t_3,n-2t_3]\cup[n-2t_3+1,n-t_3]\cup[n-t_3+1,M]\cup[M+1,n]$.
We have
\begin{multline*}
n + 1 + \left\lceil \frac{n}{2} \right\rceil = |A_1|+|A_2|+|A_3|\leq 0+3t_3+3\floor{\frac{n-4t_3+1}{2}}+0+3 \big( M-(n-t_3) \big)+2(n-M) \\
=3\ceil{\frac{n}{2}}-(n-M).
\end{multline*}
Therefore, $n$ is odd, and thus, we have $M = n$.
Since there is no rainbow solution and $t_3+t_3=2t_3$, we have $2t_3\notin A_3$. Hence by Observation~\ref{Observation:A_3iff}, we have $n-2t_3\in A_3$. Since there is no rainbow solution and $2t_3 - 1 \in A_2$, this implies $n-1 \notin A_1$. Since $[n - t_3 + 1, n] \subseteq A_2 \subseteq A_1$, we have $t_3=1$. Since $n-M < t_3$, we also have $M=n$. Then, Observation~\ref{Observation:A_3iff} implies $A_1=A_2=A_3$. Thus, we have $\abs{A_1} = \abs{A_2} = \abs{A_3} =\frac{n+1}{2}$. Since there is no rainbow solution and $1\in A_3$, no two consecutive integers are in $A_1 = A_2$. Therefore, we have $A_1=A_2=A_3=\{i\in[n]: 2|(i+1)\}$.

Now let us assume $t_3>\frac{n}{2}$. By Observation~\ref{Observation:A_3iff}, we have $A_3=[t_3,M]$. 
Then $|A_3|=M-t_3+1$, and thus, \eqref{equation:m=3_A_3} implies $t_3=\floor{\frac{n}{2}}+1$.
Then $A_3 \subseteq A_2$ and $[M,n]\subseteq A_2$ implies $\left[ \floor{\frac{n}{2}}+1,n \right]\subseteq A_2$.
Since there is no rainbow solution, we obtain $A_1 \cap \left[ \ceil{\frac{n}{2}}-1 \right]=\emptyset$. 
If $n$ is odd, this implies $A_1=A_2= \left[ \frac{n+1}{2},n \right]$, thus by maximality, we also have $A_3=A_1$, and we are done.
If $n$ is even, then by maximality, we have $A_1= \left[ \frac{n}{2},n \right]$, and either $A_2=A_1$ and $A_3=[\frac{n}{2}+1,n-1]$, or $A_2=[\frac{n}{2}+1,n]$ and $A_3=A_2$.

\section{Concluding Remarks}
In this paper, we studied the problems of avoiding rainbow solutions to the generalized Schur equation $x_1+x_2+\dots+x_m=x_{m+1}$. It would be interesting to obtain such results for other equations.  Furthermore, it would be natural to consider the measure theoretic version of our result.  Namely, let $A_1,A_2,\dots,A_k$ be measurable subsets of $[0,1]$.  Then one could consider the maximum sum and product of the Lebesgue measure of the sets $A_i$ provided that we avoid a rainbow solution of the Schur equation.

\section*{Acknowledgments}
The research of Gy\H{o}ri, Salia and Tompkins was supported by the National Research, Development and Innovation Office NKFIH, grants K132696 and SNN-135643. 
The research of Zhu was supported by NSFC grant 12401454, Basic Research Program of Jiangsu Province (BK20241361), Jiangsu Funding Program for Excellent Postdoctoral Talent grant
2024ZB179 and State-sponsored Postdoctoral Researcher program under number GZB202409. The research of He was supported by the National Natural Science Foundation of China, grants 12401445, Beijing Natural Science Foundation, grants 1244047 and China Postdoctoral Science Foundation, grants 2023M740207.

\bibliographystyle{abbrv}
\bibliography{references.bib}

\begin{thebibliography}{10}

\bibitem{02221759}
A.~Baltz, P.~Hegarty, J.~Knape, U.~Larsson, and T.~Schoen.
\newblock The structure of maximum subsets of $\{ 1, \ldots ,n \}$ with no solutions to $a+b=kc$.
\newblock {\em The Electronic Journal of Combinatorics}, 12:R19, 2005.

\bibitem{bollobas2004multicoloured}
B.~Bollob{\'a}s, P.~Keevash, and B.~Sudakov.
\newblock Multicoloured extremal problems.
\newblock {\em Journal of Combinatorial Theory, Series A}, 107(2):295--312, 2004.

\bibitem{budden2024rainbow}
M.~Budden and B.~Landman.
\newblock Rainbow numbers for the generalized {S}chur equation {$x_1+x_2+\ldots+x_{m-1}=x_m$}.
\newblock {\em Australian Journal of Combinatorics}, 90(1):66--79, 2024.

\bibitem{chakraborti2022rainbow}
D.~Chakraborti, J.~Kim, H.~Lee, H.~Liu, and J.~Seo.
\newblock On a rainbow extremal problem for color-critical graphs.
\newblock {\em Random Structures \& Algorithms}, 64(2):460--489, 2024.

\bibitem{chung1996maximum}
F.~R.~K. Chung and J.~L. Goldwasser.
\newblock Maximum subsets of (0, 1] with no solutions to {$x+ y= kz$}.
\newblock {\em The Electronic Journal of Combinatorics}, 3(1):R1, 1996.

\bibitem{chung1997integer}
F.~R.~K. Chung and J.~L. Goldwasser.
\newblock Integer sets containing no solution to $x+ y= 3z$.
\newblock In {\em The Mathematics of Paul Erd{\H{o}}s I}, pages 218--227. Springer, 1997.

\bibitem{erdHos1954problem}
P.~Erd{\H{o}}s.
\newblock On a problem of {S}idon in additive number theory.
\newblock {\em Acta Scientiarum Mathematicarum}, 15(3--4):255--259, 1954.

\bibitem{glasscock2011sumset}
D.~Glasscock.
\newblock Sumset estimates in abelian groups.
\newblock Master's thesis, Central European University, 2012.

\bibitem{hancock2017solution}
R.~Hancock and A.~Treglown.
\newblock On solution-free sets of integers.
\newblock {\em European Journal of Combinatorics}, 66:110--128, 2017.

\bibitem{hancock2016solution}
R.~Hancock and A.~Treglown.
\newblock On solution-free sets of integers {II}.
\newblock {\em Acta Arithmetica}, 180:15--33, 2017.

\bibitem{gyHori2022some}
Z.~He, P.~Frankl, E.~Gy{\H{o}}ri, Z.~Lv, N.~Salia, C.~Tompkins, K.~Varga, and X.~Zhu.
\newblock Extremal results for graphs avoiding a rainbow subgraph.
\newblock {\em The Electronic Journal of Combinatorics}, 31(1):P1.28, 2024.

\bibitem{hegarty2007extremal}
P.~Hegarty.
\newblock Extremal subsets of $\{1, \ldots, n\}$ avoiding solutions to linear equations in three variables.
\newblock {\em The Electronic Journal of Combinatorics}, 14:R74, 2007.

\bibitem{hilton1977intersection}
A.~J.~W. Hilton.
\newblock An intersection theorem for a collection of families of subsets of a finite set.
\newblock {\em Journal of the London Mathematical Society}, s2-15(3):369--376, 1977.

\bibitem{keevash2004multicolour}
P.~Keevash, M.~Saks, B.~Sudakov, and J.~Verstra{\"e}te.
\newblock Multicolour {T}ur{\'a}n problems.
\newblock {\em Advances in Applied Mathematics}, 33(2):238--262, 2004.

\bibitem{matolcsi2013sets}
M.~Matolcsi and I.~Z. Ruzsa.
\newblock Sets with no solutions to $x+y=3z$.
\newblock {\em European Journal of Combinatorics}, 34(8):1411--1414, 2013.

\bibitem{roth1953certain}
K.~F. Roth.
\newblock On certain sets of integers.
\newblock {\em Journal of the London Mathematical Society}, s1-28(1):104--109, 1953.

\bibitem{ruzsa1993solving}
I.~Z. Ruzsa.
\newblock Solving a linear equation in a set of integers {I}.
\newblock {\em Acta Arithmetica}, 65(3):259--282, 1993.

\bibitem{ruzsa2}
I.~Z. Ruzsa.
\newblock Solving a linear equation in a set of integers {II}.
\newblock {\em Acta Arithmetica}, 72(4):385--397, 1995.

\end{thebibliography}

\end{document}